\documentclass{amsart}

\usepackage{amsthm}
\usepackage{amsmath, amssymb, amsthm, amscd, amsfonts,  mathrsfs} 
\usepackage{color}

\usepackage{euscript,graphicx,color}

\newtheorem{theorem}{Theorem}

\newtheorem*{definition*}{Definition}

\newtheorem{lemma}[theorem]{Lemma}

\newtheorem{prop}[theorem]{Proposition}

\newtheorem{proposition}[theorem]{Proposition}

\newtheorem*{theorem*}{Theorem}

\newtheorem*{prop*}{Proposition}

\newtheorem{conjecture}[theorem]{Conjecture}

\newtheorem{observation}{Observation}

\newtheorem*{rem*}{Remark}

\theoremstyle{definition}

\theoremstyle{Theorem A}
\theoremstyle{Theorem B}

\newtheorem*{thmA}{Theorem A}
\newtheorem*{thmB}{Theorem B}

\DeclareMathOperator{\Crit}{{\rm Crit}}
\DeclareMathOperator{\diam}{{\rm diam}}

\DeclareMathOperator{\HD}{{\rm HD}}
\DeclareMathOperator{\hyp}{{\rm hyp}}

\newcommand{\N}{\mathbb{N}}
\newcommand{\C}{\mathbb{C}}

\def\ov{\overline}
\def\un{\underline}

\def\1{1\!\!\text{{\rm 1}}}

\title[Dimension of Julia sets]{On Hausdorff dimension of polynomial not totally disconnected Julia sets}

\author[F. Przytycki]{Feliks Przytycki$^\dag$} \address{$\dag$ Institute of Mathematics, Polish Academy of Sciences, ul. \'{S}niadeckich 8, 00-656 Warszawa, Poland}
\email{feliksp@impan.pl}

\author[A. Zdunik]{Anna Zdunik$^\ddag$} \address{$^\ddag$ Institute of Mathematics, University of Warsaw, ul. Banacha 2, 02-097 Warszawa, Poland}
\email{A.Zdunik@mimuw.edu.pl}

\thanks{The research of A. Zdunik was supported in part by the National Science Centre, Poland, grant no 2018/31/B/ST1/02495.}
\thanks{The research of F. Przytycki was supported in  part by the National Science Centre, Poland, grant no 2019/33/B/ST1/00275.}
\begin{document}
\maketitle
\begin{abstract}
 We prove that for every polynomial of one complex variable
of degree at least 2 and Julia set not being totally disconnected
nor a circle, nor interval, Hausdorff dimension of this Julia set is larger than 1.
Till now this was known only in the connected Julia set case.

We give also an example of a polynomial with non-connected but not totally disconnected 
Julia set and such that all its 
 components comprising of more than single points are analytic arcs, thus resolving a question by
  Christopher Bishop, who asked whether every
such component must have Hausdorff dimension larger than 1.
 \end{abstract}

\section{Introduction}\label{intro}

In the sequel we refer to connected components of subsets of the complex plane $\mathbb C$
that are not singletons (sets comprising of single points) as non-trivial components.

\medskip

Christopher  Bishop in \cite{bishop}, commenting on his result on the existence of an entire transcendental function with
one-dimensional Julia set, asked the following question:

\emph{The connected components of the Julia set constructed in this paper are all
either points or continua of Hausdorff dimension one.(...)
However, the situation for polynomials is
open. If a polynomial Julia set is connected, then it is either a generalized
circle/segment or has Hausdorff dimension strictly greater than 1 (this
follows from work of Zdunik [51] and Przytycki [36]). Is this also true of
the non-trivial connected components when the Julia set is disconnected?
In other words, if J(p) is disconnected, its every connected component is
either a point or a set of Hausdorff dimension strictly greater than 1?}

\

We recall that for every polynomial $p$ of degree at least 2, the Julia set $J(p)$
is the boundary of the basin of attraction to $\infty$ under the action by $p$.

\

In this article, we  answer the above   questions and related ones.
We start with key

\begin{theorem}\label{thm:main}
Let $f:\mathbb C\to\mathbb C$ be a polynomial of degree $d\ge 2$. If  the Julia set $J(f)$ is disconnected then
every non-trivial connected component $J'$ of $J(f)$ is eventually periodic, i.e. $f^\ell(J')$ is periodic with period $k$ for some integers $\ell, k$. Furthermore, either
$f^\ell(J')$  is an analytically
embedded interval (i.e. analytic arc),
 and $f^k$ on it is analytically conjugate to \;  $\pm$ Chebyshev polynomial, or $J'$ has Hausdorff dimension,
 and even hyperbolic dimension, greater than $1$.
\end{theorem}

This theorem is proved in Section~\ref{proof_thm1}. An example is presented at Figure \ref{basilicas}.

\begin{figure}[h]
\includegraphics[height=6cm]{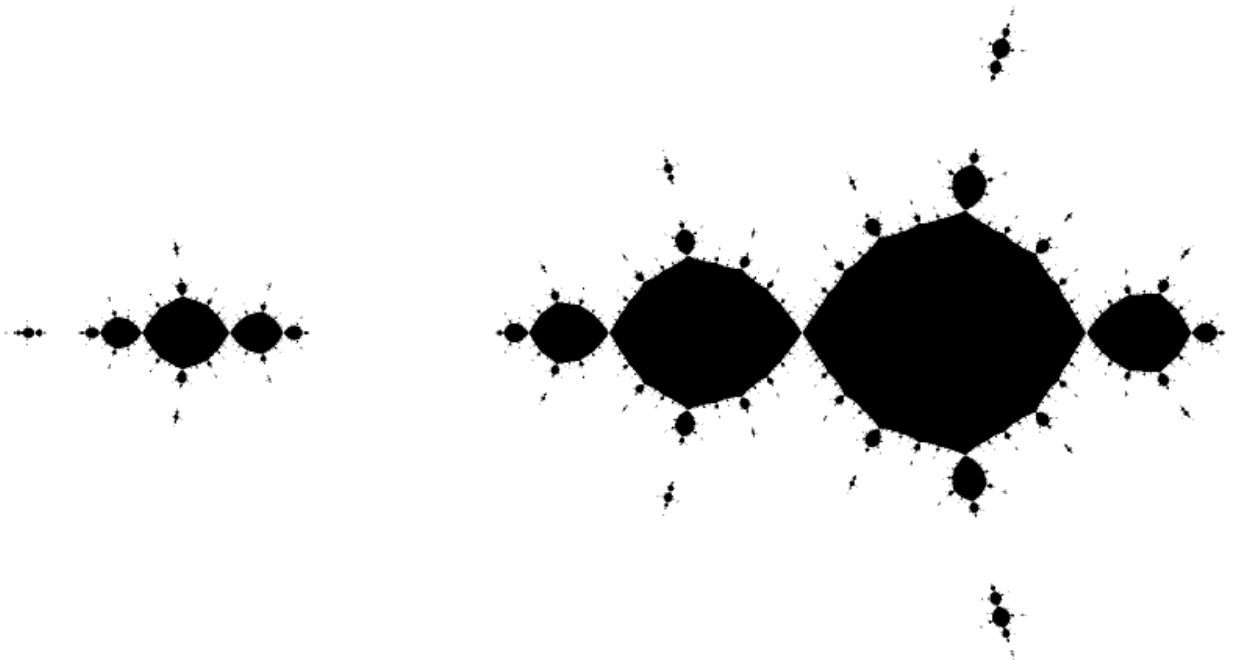}
\caption{Mandelbrot's basilicas: disconnected Julia set for $f(z)= a z^3 +z^2 - b$ with $a\approx .215$
and
$b\approx 1.455$. The critical point 0 has period 2. It is inside the biggest basilica which has $J'$
as its (fractal) boundary. Compare $f=f_{\varepsilon,\beta}$ in Section \ref{sec:example}, where $f^2(0)$ is a fixed point, an end point of $J'$ which is an interval there.}
\label{basilicas}
\end{figure}

In Section ~\ref{above} we complete the picture, including possible analytic components and proving the following theorem, which is the main result of this note.

\begin{theorem}\label{thm:HDabove} Let $f:\mathbb C\to\mathbb C$ be a polynomial of degree $d\ge 2$,  such that $J(f)$ is  not  totally disconnected. Assume also that $f$
is not complex affine conjugate to a map $z\mapsto z^d$ or a $\pm$ Chebyshev polynomial.
Then $\HD(J(f))>1$. Even more, the hyperbolic dimension, $\HD_{\hyp} (J(f))$ is greater than $1$.
\end{theorem}

\noindent

Recall that  hyperbolic dimension of the Julia set $J(f)$, analogously of any other $f$-invariant subset of $J(f)$,
$\HD_{\rm hyp}(J(f))$ is  defined as supremum of Hausdorff dimensions of isolated hyperbolic forward invariant subsets of $J(f)$. A compact set $L\subset \mathbb C$ is hyperbolic (or expanding) forward invariant for $f$ if $f(L)\subset L$ and there exists $n\in \mathbb N$ such that $|(f^n)'(z)|>1$ for all $z\in L$. The set $L$ is said to be \emph{isolated} (or repelling) if there exists its neighbourhood $U$ in $J(f)$ such that if a trajectory $(f^j(z))_{j\ge 0}$ is in $U$, then it is in $L$.

\noindent Thus, the inequality
$$\HD(J(f))\ge \HD_{\rm hyp}(J(f))$$
holds trivially.
See \cite[Chapter 12]{pubook} for the discussion on equivalent definitions of the hyperbolic dimension,
in particular the one we shall use here, introduced in Proposition \ref{prop:bowen}.

\smallskip

So, in view of Theorem \ref{thm:HDabove}, if $J(f)$ is disconnected but it has a non-trivial connected component, then $\HD_{\hyp} (J(f))>1$, even
if the dimension of every non-trivial component is equal to one.
In the latter case the dimension larger than $1$ is achieved
due to a Cantor set of other components.
To prove it, we follow a strategy by Irene
Inoquio-Renteria and Juan Rivera-Letelier in \cite{InoRiv}, where they proved
that geometric pressure $P(t)$  is larger than $-t$ times Lyapunov exponent of any probability
invariant measure on the Julia set $J(f)$ of a rational function on the Riemann sphere
$f:\widehat{\mathbb{C}}\to\widehat{\mathbb{C}}$ (for the definition see Section \ref{proof_thm1}, provided the exponent is positive. See Proposition~\ref{prop:bowen} for the definition of the geometric pressure and
Remark 12.5 in \cite{pubook}.

Theorem~\ref{thm:HDabove} strengthens the result of \cite{zdunik2}, where the analogous theorem was proved for polynomials with connected Julia set.

\

Finally, we show in Section ~\ref{sec:example} that the situation we deal with in the proof of Theorem~\ref{thm:HDabove} really may happen:
in Proposition~\ref{prop:interval_components} we provide an example of a family of polynomials of degree $3$ with
non-trivial  components of the  Julia set  $J(f)$, each of which is an analytically embedded interval.
This resolves Christopher Bishop's question quoted above in the negative.

\

{\bf Remark.} \ One can also consider \emph{radial} (or \emph{conical}) Julia set $J_{\rm r}(f)\subset J(f)$ which consists of all those points $z\in J(f)$  for which there exists $\delta>0$ such that for infinitely many $n\in \mathbb N$ the map $f^n$ admits a holomorphic inverse branch $f^{-n}:B=B(f^n(z),\delta)\to \mathbb C$ sending $f^n(z)$ to $z$, for the Euclidean balls $B$ in the case  $f$ is a polynomial and balls in the spherical metric for $f$ a rational function.
Then
$\HD_{\hyp} (J(f))=\HD(J_{\rm r}(f))$, see e.g. \cite[Section 3]{dfsu} for this, in the setting including also rational and meromorphic functions, and see references therein; other versions of conical Julia sets for rational functions were introduced in \cite{feliks_conical}. Therefore in Theorem \ref{thm:HDabove} we can write
\begin{equation}
\HD(J_{\rm r}(f))>1.
\end{equation}

\bigskip

\section{Basic notions, repelling boundary domains and proof of Theorem~\ref{thm:main}}\label{proof_thm1}

For a polynomial $f:\mathbb C\to\mathbb C$ of degree $d\ge 2$ one considers the set of points escaping to infinity
$$
A_\infty(f):=\{z\in \mathbb C: f^n(z)\to \infty \ {\rm as}\ n\to\infty\}.
$$
Here $f^n=f^{\circ n}=f\circ f ...\circ f$ $n$-times the composition of $f$. It is clear that this is exactly the set of points with the forward orbits $(f^n(z))_{n=0,1,...}$ unbounded. The set $A_\infty(f)$ is called the \emph{basin of attraction to infinity}. It is open and connected (since otherwise its bounded component would contain a pole for an iterate of $f$ by Maximum Principle). Define the filled-in Julia set: $K(f):=\mathbb C \setminus A_\infty(f)$ and Julia set $J(f): =\partial K(f)=\partial A_\infty(f)$ where $\partial$ means the boundary,  as in the Introduction. Clearly both $K(f)$ and $J(f)$ are completely invariant, namely $f^{-1}(K(f))=K(f)$ and $f^{-1}(J(f))=J(f)$.

A point $z\in\mathbb C$ is said to be \emph{critical} (or $f$-\emph{critical}) if
the derivative $f'(z)$ is equal to 0. The basin $A_\infty(f)\cup \infty$ is simply-connected, that is $J(f)$ be connected, if and only if the forward trajectories of all critical points are bounded.
At the other extreme if all critical points escape to the infinity then  $J(f)$ is totally disconnected. However the opposite implication is false!. For more details see e.g. \cite[Section III.4.]{cg}, and \cite{Brolin} in particular Theorem 3.8 and an example following it.

Another definition of Julia set is that it is the complement of Fatou set
${\tt F}(f)$ on which the sequence
$f^n$ is normal, that is for every its subsequence and a compact set $K\subset {\tt F}(f)$  a   subsubsequence of it is uniformly convergent on $K$. This definition is valid also for
a rational function $f:\widehat{\mathbb C}\to \widehat{\mathbb C}$ 
acting on the Riemann sphere $\widehat{\mathbb C}=\mathbb C \cup \infty$.
The Julia set $J(f)$ is compact, completely $f$-invariant, namely $f^{-1}(J(f))=J(f)$, and a closure of repelling periodic orbits contained in it. For entire transcendental functions $f$ the definition is the same.

Intuitively, the action of $f$ on $J(f)$
is `chaotic'  and at points in $J_{\rm r}(f)$ yields a similarity of infinitesimal parts of $J(f)$ with its big pieces, thus exhibiting a fractal local geometry/nature of $J(f)$.



\bigskip

We go back to a polynomial $f$.
Choose $R$ so large that $K(f)$ is a subset of $\mathbb D(0,R))$,
the Euclidean disc with radius $R$ and origin at $0$,
and moreover
\begin{equation}\label{eq:large_R}
\overline {f^{-1}(\mathbb D(0,R))}\subset \mathbb D(0,R).
\end{equation}
Then, by its complete invariance, $K(f)\subset \bigcap_{n=0}^\infty f^{-n}(\mathbb D(0,R))$.
Directly by the definition of $K(f)$ the opposite inclusion holds, hence the equality holds.

Let $C$ be a connected component of  the filled-in Julia set $K(f)$.
Denote by $U_n(C)$ the unique connected component of $f^{-n}(\mathbb D(0,R))$ containing $C$.
The boundary $\partial U_n(C)$ is disjoint from $K(f)$ since $\partial \mathbb D(0,R)$ is, and $f(K(f))=K(f)$.
Notice also that by \eqref{eq:large_R}
\begin{equation}\label{rel-compact}
\overline{U_{n+1}(C)}\subset U_n(C).
\end{equation}
Finally notice that all $U_n$ are simply-connected (topological discs) by Maximum Principle.

\smallskip

Define
$
C':=\bigcap_{n=0}^\infty U_n(C).
$
By $U_n(C)\supset C$ we get $C'\supset C$.
On the other hand note that $C'$ is connected as the intersection
of the decreasing sequence of compact connected sets $\overline{U_n(C)}$.
It is also clear from the definition of $C'$ that it consists of non-escaping points, hence it is contained in $K(f)$. Therefore $C'\subset C$.
We conclude with
$$
C=\bigcap_{n=0}^\infty U_n(C).
$$

One fact pivotal for our paper is

\begin{thmA}[Qiu \& Yin \cite{qiu}, Kozlovski \& van Strien \cite{strien}]
For a polynomial $f$ of degree at least $2$
if a component
of filled-in Julia set $K(f)$ is not a point, then its forward orbit contains a periodic component containing a critical point.


\end{thmA}

We refer to periodic components that contain a critical point as \emph{critical periodic components}.

\smallskip

In the proof of Theorem \ref{thm:HDabove} we shall use in fact a weaker version of Theorem A, that if the Julia set of a polynomial is not totally disconnected then there exists a non-trivial critical periodic component of $K(f)$.

\smallskip

Theorem A
was proved
for degree 3 polynomials by Branner and Hubbard \cite{BraHub} and independently by Yoccoz, but for higher degree polynomials stayed not proven
for a long time.

\smallskip

Consider any non-trivial  component of $K(f)$. We can replace it by a periodic component $C$ which is its image under an iterate of $f$,  i.e., $f^k(C)=C$ for some $k\ge 1$. Since replacing $f$ by $f^k$ does not change the Julia set, we may assume and we do assume from now on that $f(C)=C$.

\begin{definition*}[polynomial-like maps]\label{def:poly_like}
A polynomial-like map is a triple $(U,U',\Phi)$, where $U$ and $U'$ are open subsets of $\mathbb C$
homeomorphic  to a disc, with $U'$ relatively compact in $U$ and $\Phi:U'\to U$ holomorphic, proper,
of degree $\deg \Phi \ge 2$.

\emph{Proper}
means that preimage of every compact set is compact. 

\emph{Degree} is the number of preimages
$\# (\Phi^{-1}(z))$ for an arbitrary $z\in U$
 counted with multiplicities.
It does not depend on $z$.
\end{definition*}

The notion \emph{polynomial-like} was introduced by Adrien Douady and John Hubbard in \cite{dh}, Section I.1.
See also \cite[Section VI.1]{cg} or e.g. \cite[Section 7.1]{bf}, where
a smoothness or analyticity of $\partial U$ were additionally assumed. This allowed to extend $\Phi$
continuously to $\partial U'$ and define \emph{proper} by $\Phi(\partial U')=\partial U$. One can always achieve this analyticity by a slight decreasing of the Douady and Hubbard $U$ and $U'$, see \cite[Remark 7.2]{bf}. In our case we adjust $R$, see below.

Note that degree above is well defined for $U'$ and $U$ being open connected domains not necessarily simply connected (topological discs). This will be used in the proof of Lemma \ref{lem:poly_like}).

\medskip

For a polynomial-like $\Phi$ the filled-in Julia set and the Julia set
can be defined in the same way as for a polynomial, namely
\begin{equation}
K(\Phi)=\{z\in U': \Phi^n(z)\in U' \ \hbox{for all}\ \ n\in\mathbb N\} \ \ {\rm and} \  \  J(\Phi)=\partial K(\Phi).
\end{equation}

\medskip


Back to our component $C$, for  an arbitrary $n\in\mathbb N$, \ $F:=f|_{U_{n+1}(C)}$ maps $U_{n+1}(C)$ onto $U_{n}(C)$ (the same $C$). So since $F$ is proper and due to \eqref{rel-compact},  the triple
$(U_{n}(C), U_{n+1}(C), F)$ is
a polynomial-like map.
 The boundaries of the domain and range are analytic if $\partial {\mathbb D}(0,R)$ is taken disjoint from the forward orbits of critical points. We shall use the latter property to guarantee that for every
 $n\in\mathbb N$,  each two distinct components of $f^{-n}({\mathbb D}(0,R))$ have disjoint closures.

\begin{lemma}\label{lem:poly_like}
If $C$ is a fixed by $f$ non-trivial critical component of $K(f)$, then for each $n$ large enough
the map $F$ defined above is polynomial-like, with
$C$ being its  filled-in Julia set.
\end{lemma}

Though $F$ is polynomial-like for each $n$, there is no reason to expect that $K(F)=C$; only $K(F)\supset C$ must hold, since forward trajectories of points in $C$ do not leave $U_{n+1}(C)$ by the forward invariance of
$C$. So Lemma \ref{lem:poly_like} is not void.

\begin{proof} Let us write here $U_n$ for $U_n(C)$ for all $n$.
We claim that for $n$ large enough
$U_{n+1}$ is  the only connected component of $f^{-1}(U_n)$ contained in $U_n$. No other component
intersects $U_n$.

\medskip

Indeed, if a component $V$ of $f^{-1}(U_n)$ intersects $U_n$, it must have its closure entirely in $U_n$.  Otherwise $\overline{V}$
 would intersect $\partial U_n$. Taking the images under $f^{n+1}$  we would conclude that
 $\overline{{\mathbb D}(0,R)}\cap f(\partial{\mathbb D}(0,R))\not=\emptyset$ contradicting \eqref{eq:large_R}.

 Suppose now, there is a component $V$ of $f^{-1}(U_n)$ with closure entirely in $U_n$, different from $U_{n+1}$. Then there is a critical point of $f$ in
 $W':=U_n\setminus (\overline{V}\cup \overline{U_{n+1}})$.
Otherwise $f:W'\to U_{n-1}\setminus \overline{U_{n}}=W$  would be a covering map from $W'$ to the annulus $W$.
This contradicts the equality $\chi(W')=\deg(f|_{W'}) \chi(W)$ for Euler's characteristics since $\chi(W')$ is negative and
$\chi(W)=0$.  See Figure \ref{critical}.

\

\begin{figure}
\includegraphics[height=5cm]{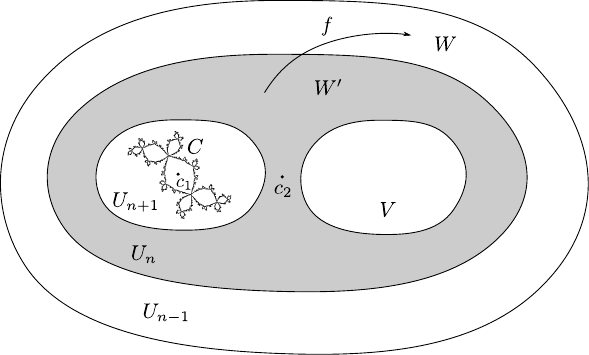} 
\caption{A capture of a critical point.}
\label{critical}
\end{figure}

\

Since the number of critical points of $f$ is finite, they cannot happen in $W'$ for $n$ large enough.
The claim follows.

\medskip


We conclude that  $F^{-k}(U_n)=U_{n+k}$ for all $k>1$, so
$\bigcap_{k=0}^\infty F^{-k}(U_n(C))=$

\noindent $\bigcap_{k=0}^\infty U_{n+k}(C)=C$.

\medskip

Notice finally that degree of $F$ is larger than 1, since $C$ hence $U_{n+1}$ is assumed to contain a critical point.


\end{proof}
{\bf Remarks} 1. Notice that the assumption that $C$ is critical need not be specified. It follows automatically from the assumption that $C$ is non-trivial. Indeed, if there were no critical point in $U_{n+1}$ then
the moduli of all annuli $U_n\setminus \overline{U_{n+1}}$ for $n$ large enough would be positive equal to each other so $C$ would be a point, contradicting the non-triviality of $C$.

2. Notice that for $F$ as in Lemma \ref{lem:poly_like},
$\deg(F)<\deg(f)$. Indeed, $C=K(F)$ is by definition completely invariant for $F$, namely $F^{-1}(C)=C$.
At the same time there exists a component
$C'$ of $f^{-1}(C)$, in $K(f)$ by its complete invariance for $f$, different from $C$,
since otherwise $C$ would be completely invariant also for $f$ so $K(f)=K(F)$.
Then $f(U_{n+1}(C'))=U_n(C)$ for all $n\in \mathbb N$, hence
$f$ maps $C'$ onto $C$ thus adding to $\deg(F)$ an additional contribution.\footnote{A more general approach to the proof, more standard, is via the \emph{topological exactness} of $f$ on $J(f)$, compare
a footnote to Comment 1 in Section \ref{final}, yielding the density of
$\bigcup_{n=0}^\infty f^{-n}(\{z\})$ in $J(f)$ for each $z\in J(f)$. This density will be used also in Proof of Proposition~\ref{prop:interval_components}.}

\bigskip

Below we recall the definition of  RB-domain, which was introduced in \cite{feliks3}, with the name introduced in \cite{puz} together with crucial applications of this notion.

\begin{definition*}[RB-domain]
Let $\Omega$ be a simply connected open domain in the Riemann sphere $\widehat{\mathbb C}$ with $\#(\widehat{\mathbb C}\setminus \Omega)>2$.
Assume there exists a holomorphic map defined on a neighborhood $U$ of  $\partial\Omega$ such that
\begin{equation}\label{eq:def_RB}
f(U\cap\Omega)\subset\Omega, \;\; f(\partial\Omega)=\partial\Omega
\quad\text{and}\quad \bigcap_{k=0}^\infty f^{-k}(U\cap \ov\Omega)=\partial \Omega.
\end{equation}

Then $\Omega$  is called a repelling boundary domain (RB-domain).
\end{definition*}

\begin{observation}\label{obs:1} The domain $\Omega:=\widehat{\mathbb C}\setminus C$   together with $f$ restricted to $U:=U_{n+1}(C)$ with
$n$ sufficiently large, as in Lemma~\ref{lem:poly_like},
is a \emph{repelling boundary domain (RB--domain)}. Indeed, by this Lemma, $f|_{U_n}^{-1}(C)=C$ which yields $f(U\cap\Omega)\subset\Omega$.
\end{observation}
\medskip

Now we rely on:

\begin{thmB}[see \cite{feliks}, Theorem A']
 If $\Omega$ is an RB-domain then either $\HD_{\rm hyp}(\partial \Omega)>1$  or  $\partial\Omega$
 is an analytic Jordan curve or an analytically embedded interval, where $f$ is topologically conjugate to $z^d$
 or to a $\pm$ Chebyshev polynomial, respectively.
\end{thmB}

Theorem B was stated as a conjecture in \cite{puz}.
A comparison of harmonic measure $\omega$ on $\partial \Omega$ to $\mathcal H^1$ which is the Hausdorff measure  in dimension 1, is the key. For $\omega$ absolutely continuous it was proved in \cite{zdunik} that $\partial \Omega$ was analytic.
 It was also claimed in \cite{zdunik} that for $\omega$ singular $\HD (\partial \Omega)>1$ could be proved by adapting the methods of
\cite{zdunik2}.
A detailed proof of this appeared in \cite{feliks}. Analogously to \cite{zdunik2}, the proof not
only showed that $\HD(\partial\Omega)>1$, but, actually, the
hyperbolic dimension of  $\partial\Omega$ for $f$ restricted to it, was larger than $1$. For a survey see \cite{P-ICM}.

\

Combining the above Observation and  Theorem B, concludes the proof of Theorem~\ref{thm:main}, except it does not
exclude
$C$             
such that $\partial C=\partial\Omega$ is an analytic Jordan curve.

So, now we consider the latter case.
Denote $\partial C$ by $\gamma$.
Denote by $\Omega'$ the bounded (internal) component of $\mathbb C\setminus \gamma$, and consider now $\Omega'$  instead of the external one  $\widehat\C\setminus C$, which was considered above.
The domain $\Omega'$  is forward invariant for $f$ since $f(C)=C$ (as we replaced $f$ by its adequate iterate) and due to Maximum Principle.

We already know that $\Omega$ is an RB--domain. Let $U$ be a neighbourhood coming from the definition of RB-domain (see\eqref{eq:def_RB}). We  know (see Observation \ref{obs:1})
that any set $U_n$, with $n$ sufficiently large can be taken as $U$ in this definition.

Since $\gamma =\partial C$ is an analytic curve, there exists a neighbourhood $W$ of $\gamma$ , for which the Schwarz Reflection Principle applies.

Define $\widehat U_n:=U_n\cap\Omega$ and choose $n$, additionally, so large that $\widehat U_n\subset W$.
Take $\widehat U:=\widehat U_n$ for this $n$.
We have by \eqref{eq:def_RB}:
$$f(U\cap\Omega)\subset\Omega, \;\; f(\partial\Omega)=\partial\Omega
\quad\text{and}\quad \bigcap_{k=0}^\infty f^{-k}(U\cap \ov\Omega)=\partial \Omega=\gamma .$$

Denoting by $U'$ the reflection of $\widehat U$ with respect to $\gamma$, we see that the same holds true with $U$ replaced by the set $U^{\rm r}$, being the union of $U', \widehat U$ and the arc $\gamma$ along which their closures intersect,  and $\Omega$ replaced by $\Omega'$, i.e.
$$
f(U^{\rm r}\cap\Omega')\subset\Omega', \;\; f(\partial\Omega')=\partial\Omega'
\quad\text{and}\quad \bigcap_{k=0}^\infty f^{-k}(U^{\rm r}\cap \ov\Omega')=\partial \Omega'=\gamma .
$$

In other words,  $\Omega'$ is also a repelling boundary domain.  Since  $f(\widehat U)$ contains
the closure of $\widehat U$ in $\Omega$, by symmetry $f(U')$ contains the closure of $U'$ in  $\Omega'$.
In other words
the compact set
$\Omega'\setminus U'$ is mapped by holomorphic $f$ into its interior.
It follows that there is an attracting fixed point ( a sink)  in $\Omega'$, see e.g. \cite[Wolff-Denjoy Theorem]{cg},  and
$\Omega'$ is
the immediate basin of attraction to this sink\footnote{``immediate'' stands for a connected component of the basin, which contains the sink}. One can also refer to the observation that (due to RB) $f$ is a contraction on $\Omega'\setminus U'$ in the hyperbolic metric on $\Omega'$\footnote{One can also argue that, as $\Omega'$ is connected bounded and invariant, it is a component of the Fatou set for $f$. As being an RB-domain it cannot be neither parabolic, nor Siegel disc, nor Herman ring, so it must be an immediate basin of attraction to a sink. See e.g. \cite[Theorem IV.2.1]{cg}.}.

Using \cite[Lemma 9.1]{Brolin} we now conclude that $\gamma=\partial C$ is a circle and $f(z)=z^d$ in some complex affine coordinates.  (See also  \cite[Theorem A]{feliks} or    \cite[Theorem 8.1]{P-ICM}). But this contradicts the assumption that $J(f)$ is disconnected.\footnote{We are grateful to  Fei Yang for bringing our attention to this absence of analytic Jordan curves.}

\medskip

{\bf Remark.} A key feature on the domain $\Omega'$ is that it is $f$-invariant, unlike
$\Omega$ containing components of $f^{-1}(C)$. So having the nonempty interior of $C$ we can apply
\cite[Lemma 9.1]{Brolin}, what does not work for $\Omega$. The RB-property is weaker than
being an immediate basin of attraction to a sink.

The common feature is the repulsion
of the boundary to the side of a domain $\Omega$, which for a pullback $g=R^{-1}\circ f\circ R$ of $f$
 for a Riemann mapping $R:{\mathbb D}(0,1) \to \Omega$ allows to extend $g$ holomorphically beyond $\partial{\mathbb D}(0,1)$, expanding,  see \cite[Section ``Resolving Singularities'']{feliks3}, thus allowing easily to invoke Gibbs measures, see e.g. \cite{P-ICM}.

\section{Proof of Theorem~\ref{thm:HDabove}}\label{above}

We shall use a version of Bowen's formula, which can be found in \cite{feliks_conical},
see also \cite[Section 12.5]{pubook},  for a strengthened version.
\begin{prop}\label{prop:bowen}
Let $f$ be a rational map of degree $d\ge 2$. There exists an exceptional set $E\subset \widehat{\mathbb C}$ of Hausdorff dimension $0$, such that for every $t\ge 0$ (even every real $t$)
and for every $z\notin E$ the limit
$$
P(t,f;z)=\lim_{n\to\infty}\frac{1}{n}\log\sum_{v\in f^{-n}(z)}\frac{1}{|(f^n)'(v)|^t}
$$
exists, and is independent of $z\in\widehat{\mathbb C}\setminus E$. Denote the common value as $P(t,f)$. It is called the geometric pressure.  The function $t\mapsto P(t,f)$ is continuous and non-increasing. Moreover $P(0,f)$ is positive
(equal to $\log d>0$, the topological entropy).
The following formula holds:

$$\HD_{\rm hyp}(J(f))=\inf\{t>0: P(t,f)\le 0\}.$$
\end{prop}
\noindent In words, $\HD_{\rm hyp}(J(f))$ is the first zero of $P(t,f)$.

There is an abundance of points not belonging to $E$ (i.e. non-exceptional). Those are for example all points which are not post-critical (not in the forward $f$-trajectory of an $f$-critical point)
and which belong to  basins of attraction to periodic orbits.

\bigskip

In view of Theorem~\ref{thm:main}, to prove Theorem~\ref{thm:HDabove} it remains only to prove the following 

\begin{proposition}\label{prop:dim_large}
Let $f:\C\to \C$ be a polynomial of degree $d\ge 2$ with disconnected Julia set. Suppose there exists a periodic
connected component $C$ of $K(f)$
being
an analytic arc.
Then $\HD_{\hyp} (J(f))>1$.
\end{proposition}


\

\begin{proof}[Proof of Proposition~\ref{prop:dim_large}]
We shall prove that
 $P(1,f)>0$, which implies, by Proposition \ref{prop:bowen}, that
 $\HD_{\hyp} (J(f))>1$.
Consider
 $F=f:U_{m+1}(C)\to U_m(C)$, with some fixed $m$  large enough, so that $F$ is a polynomial-like map and $C$ its
 Julia set, here equal to the arc $K(F)$,
 see Lemma \ref{lem:poly_like} (there $m$ was denoted $n$).

\smallskip

For $x\in U_m\setminus C $ denote
\begin{equation}\label{eq:L_n}
L_n(F,x):=
\log \sum_{F^n(y)=x} \exp(-\log |(F^n)'(y)| )=
\log \sum_{F^n(y)=x} |(F^n)'(y)|^{-1}.
\end{equation}

\smallskip
To continue the proof of Proposition~\ref{prop:dim_large} we need two lemmas (Lemma~\ref{lem:key} and Lemma~\ref{lem:other_branch}) , which we formulate and prove below. 

\begin{lemma}\label{lem:key}
There exists $C_0 > 0$ such that for all $x\in U_m\setminus C$ and $n>0$\,
\begin{equation}\label{key}
L_n(F,x) \ge -C_0.
\end{equation}
\end{lemma}

\begin{proof}
Let $\Omega= \widehat\C\setminus C$.
We know from the proof of Theorem~\ref{thm:main} that $\Omega$ is an RB-domain.
Recall that $C=\partial \Omega$ is an analytically embedded interval,
as assumed in Proposition \ref{prop:dim_large}.
We recall the final step of the proof of the analyticity of $C$ provided in \cite[Proposition 6]{zdunik}), which will allow us to deduce also the inequality \eqref{key}.

So assume
that $C$ is an analytically embedded interval, with endpoints, say $-1, 1$.
Now we use Zhukovsky's function, namely the
ramified (branched) covering map $\Pi:\widehat{\mathbb{C}} \to \widehat{\mathbb{C}}$, ramified over $-1$ and $1$:  $\Pi(z)=\frac{1}{2}(z+\frac{1}{z})$
and the preimage $\gamma=\Pi^{-1}(C)$, see Figure \ref{Zhukovsky}. Then $\gamma$ is an analytic Jordan curve, as proved
in \cite[Proposition 6]{zdunik}.\!\footnote{The property of a closed continuous  arc joining $-1$ and $1$, saying that the Jordan curve being its preimage under $\Pi$ is analytic, can be assumed to be a definition of the analyticity of the arc
(called therefore an analytic arc or an analytically embedded interval).}


\bigskip

The curve $\gamma$ divides the
sphere into two disc $\mathcal D_i, i=1,2$ and $\Pi(\mathcal D_1)=\Pi(\mathcal D_2)=\widehat{\mathbb
C}\setminus C$. The map $F:U_{m+1}\setminus C \to U_m\setminus C$ can be lifted to holomorphic maps $G_i$ defined on the respected components $W_i:=\mathcal D_i\cap \Pi^{-1}(U_{m+1}\setminus C))$,
so that
$G_i(W_i)\subset\mathcal D_i$. Next consider just one of these components,
say $W_1$\footnote{\emph{A priori} we do not know that both $G_1$ and $G_2$ are mutual holomorphic extensions through $\gamma$, but we need only one of them.}.

\

\begin{figure}
\includegraphics[height=5cm]{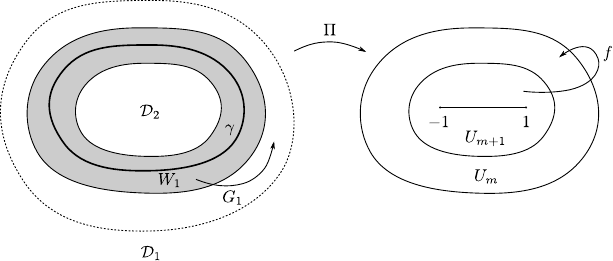} 
\caption{A lift by Zhukovsky function.}
\label{Zhukovsky}
\end{figure}

\


By Carath\'eodory's Theorem (local version), see e.g. \cite[Ch.II.3, Theorem 4']{Goluzin}, and Schwarz reflection principle, we can extend $G_1$ holomorphically from $\mathcal D_1$ to a map $\widehat G$ acting on a neighbourhood of $\gamma$. Then the map $\widehat{G}$ is expanding on $\gamma$
because of uniform convergence  of $\widehat G^{-n}$ on a neighbourhood of $\gamma$ to $\gamma$ which is a nowhere dense set. This implies  that the limit functions are constant, hence uniform
 convergence of derivatives $(\widehat G^{-n})'$ to 0, by normality, see e.g. \cite[Section 7]{feliks3} or \cite[Proof of Theorem  6.5]{Brolin}.
 Due to the just mentioned expanding property, the domain $\mathcal D_1$ bounded by $\gamma$ is an
RB-domain for the action of $\widehat G$. (It is seen also directly, by
  taking as a neighbourhood of $\gamma=\partial\mathcal D$ in the definition of RB, the set $\widehat U=\Pi^{-1}(U_m)$.


\bigskip

{\bf  Remark.}
Here is a different argument for the existence of the extension of $G$ from $W_1$ beyond $\gamma$ not using Schwarz reflection principle: Consider the lift of $F$ on say $U_{m+1}$ to $G$ on the set $W:=\Pi^{-1}(U_{m+1})$, which is a neighbourhood of $\gamma$ (the both sides and $\gamma$ itself), for the branched covering $\Pi$, directly by the formula
\begin{equation}\label{lift}
G:= \Pi^{-1} \circ F \circ \Pi.
\end{equation}
Because of $\Pi^{-1}$ it is not \emph{a priori} clear that $G$ in this formula is well defined.
 Here is an explanation. First define $G:W_1\to \mathcal D_1$ as $G_1$ above. Fix an arbitrary $z\in W_1$ and denote $w:=G(z)$. Next extend  $G$ using \eqref{lift} along curves in $W$ starting from $z$ mapped to $w$, using \eqref{lift}, the curves omitting
$$
A:=\Pi^{-1}(\Crit (F) \cup \{-1,1\}),
$$
where $\Crit(F)$ is the set of $F$-critical points.
Notice that $A$ is precisely the set mapped by $F \circ \Pi$ to $1$ or $-1$, the critical values of $\Pi$, namely where
$\Pi^{-1}$ has singularities. This is so because $F$ is topologically conjugate to $\pm$ Chebyshev polynomial. So all the singularities of $G$ defined by \eqref{lift} are holomorphically removable.
\noindent Notice finally that $W$  is a topological annulus, with the fundamental group generated by the homotopy class of any curve $\gamma '$ in $W_1$ starting and ending at $z$ running once along $W_1$.
 The growth of the prolongation $G$ along $\gamma'$ is 0, because it starts and ends at $w$.
 It is so because $\gamma'\subset W_1$, where $G$ has been well-defined. 
\noindent So $G$ being an analytic continuation along curves is a well-defined (single-valued) holomorphic function on $W$.
In fact this $G$ coincides with the extension $\widehat G$ found before, since these maps are holomorphic and coincide in $\mathcal D_1$.

\bigskip

Let us return to the proof of Lemma \ref{lem:key}.
Let $x\in U_m\setminus C$. Our aim is  to find a lower bound for $L_n(F,x)$.
Denote by $w$ the unique preimage of $x$ under $\Pi$ in $\mathcal D_1$.
The map $\Pi$ gives a bijection between the sets $\{v\in G^{-n}(w)\}$ and $\{y\in F^{-n}(x)\}$.
Putting $y=\Pi(v)$ we obtain:
$$
(F^n)'(y)=(G^n)'(v)\cdot\frac{\Pi'(w)}{\Pi'(v)}
$$
and, consequently,
\begin{equation}\label{transfer2}
\sum_{y\in F^{-n}(x)}\frac{1}{|(F^n)'(y)}=\frac{1}{|\Pi'(w)|}\cdot
\sum_{v\in G^{-n}(w)}\frac{1}{|(G^n)'(v)|}\cdot |\Pi'(v)|.
\end{equation}

\bigskip

Let us  estimate  the expression above, but without  $\Pi'$,
that is let us estimate $L_n(G, w)$ (defined for $G$ as we did for F in
\eqref{eq:L_n}).
First assume $w\in \gamma$.
Denote by ${\tt Le}$ the Euclidean length of curves contained in our analytic curve $\gamma$.
There exists a constant $c>0$ such that for each $n$ and every $w\in \gamma$,
choosing an arbitrary $w'\in \gamma\setminus\{w\}$, denoting by $\gamma(w)$
the open arc  $\gamma\setminus \{w'\}$ we have
\begin{equation}\label{transfer1}
\sum_{G^n(v)=w} |(G^n)'(v)|^{-1}\ge c \cdot \sum_{G^n(v)=w} {\tt Le}(G_v^{-n}(\gamma(w)))/{\tt Le}(\gamma(w))\ge
\end{equation}
$$
c \cdot {\tt Le} (\gamma)/{\tt Le}(\gamma)=c>0.
$$
This is due to bounded distortion for iterates (or by Koebe distortion lemma) by a constant $c$,
see \cite[Section 6.2]{pubook}. The subscript $v$ at $G^{-n}$ means the component containing $v$.
Due to bounded distortion $w\in \gamma$ can be replaced by an arbitrary point $w$ in a neighbourhood of $\gamma$ and the same estimate as in \eqref{transfer1} is achieved for for such points.

\smallskip

We could define   so-called transfer operator (Perron-Frobenius-Ruelle) for potential $\psi:=-\log |G'|$
acting on a continuous function $\varphi$
$$
\mathcal L_\psi (\varphi)(w):=\sum_{G(v)=w} (\exp \psi(w)))\varphi(w).
$$
Then the expression estimated in \eqref{transfer1} could be written as $\mathcal L^n_\psi(\1)$.
Compare \eqref{eq:L_n}.

To continue \eqref{transfer2} we need to estimate from below the value
$\mathcal L_\psi^n(\varphi)(w)$ for potential $-\log |G'|$, where $\varphi=|\Pi'|$.
The function $\varphi$ has value zero at two points, but $\mathcal L_\psi^2(\varphi)$ is  positive, thus bounded from below  by some constant $c_1>0$.
So,
$$
\mathcal L_\psi^n(\varphi)(w)=\mathcal L_\psi^{n-2}(\mathcal L_\psi^2(\varphi))(w)\ge \mathcal L_\psi^{n-2}(c_1\cdot \1)(w)=c_1\cdot \mathcal L_\psi^{n-2}(\1)(w),
$$
and the last term is bounded below by $c>0$ as in \eqref{transfer1}.

Since $\frac{1}{\Pi'(w)}$ is bounded away from 0 in $\Pi^{-1}( U_m)$, since $\Pi'$ is upper bounded, we are done.

\end{proof}

For all $n=0,1,...$ write $F_n:=f|_{U_{n+1}(C)}$ the polynomial-like mappings considered in
Section \ref{proof_thm1}.
Consider now an arbitrary $m\in\mathbb N$ such that
$C$ is the filled-in Julia set for $F_m$, existing by Lemma~\ref{lem:poly_like}.

\begin{lemma}\label{lem:other_branch}
There exists an integer $N_1>0$ and a
connected component $V$ of $f^{-N_1}(U_m)$, such that $\overline V\subset U_m\setminus U_{m+1}$.
\end{lemma}
\begin{proof}

Notice that there exists $k: 0\le k \le m$ for which
there exists a component
$U' $ of $f^{-1} (U_{m-k})$ in $U_{m-k}$, different from $U_{m+1-k}$.
Otherwise $C=K(f)$ which is not the case as $C$ is a connected component and $K(f)$ is disconnected.
Compare Proof of Lemma~\ref{lem:poly_like}.

Consider
$V'$, an arbitrary component of $(f^k|_{U_m})^{-1}(U')$. It is a subset of $U_m$ and
a component of $f^{-(m+1)}(\mathbb D(0,R))$.
It is disjoint
from $U_{m+1}$, hence bounded away from $C$.

Note that $f^{m+1}$ maps $V'$ onto the disc $\mathbb D(0,R)$ and
$U_m \subset \mathbb D(0,R))$.  Select
$V$, an arbitrary component of $f^{-(m+1)}(U_m)$ in $V'$.  It satisfies the assertions of our Lemma
with $N_1:=m+1$.



\end{proof}

\

Notice that $V$ intersects $J(f)$ because $U_0$ does and $J(f)$ is completely invariant for $f$.
More precisely $V$ contains branched holomorphic images of $C$, namely components
of $f^{-N_1}(J(f))$. On Figure~\ref{basilicas} these are small 'basilicas' surrounding the big critical fixed one.


\

Having Lemmas~ \ref{lem:key} and \ref{lem:other_branch} at our disposal, we continue the proof of Proposition ~\ref{prop:dim_large}.
In order to simplify   the notation, we pass again to the iterate of $f$, replacing now $f$ by $f^{N_1}$ and $F$ by $F^{N_1}$.
Recall again that this modification does not change the Julia set.
To simplify further the  notation, denote
  $U_0:=U_m(C)$, $U_1:=U_{m+N_1}(C)$.
So, from now on, after this modification, the component $V$ becomes just a component of $f^{-1}(U_0)$ in $U_0$ different from $U_1=U_1(C)$. We denote $f|_V$ by $F_V$. This pairs up with $F$ on $U_1$
denoted also as $F_{U_1}$.

\bigskip

Now we build an infinite collection of multivalued maps $\phi_\ell$, $\ell\in\mathbb N$,  as follows:
Let us note that $F_V:V\to U_0$ is a holomorphic proper map onto $U_0$. Since the degree  of
this map may be larger than one, the inverse may be not well defined. However,
we  shall  use the notation $h: U_0\to V$ to denote the multivalued
inverse of $F_V$.
Thus,  $h$  assigns to a point $x\in U_0$ a collection 
of its
preimages for $F_V$.

\smallskip

Consider a family of multivalued  maps $\phi_\ell:U_0\to U_0$, $\ell=0,1,\dots$,
defined as  multivalued (branched) inverses
\begin{equation}\label{multivalued}
\phi_\ell:=F^{-\ell} \circ h
\end{equation}
where $F^{-\ell}:U_0\to U_\ell$
are multivalued holomorphic maps given by  multivalued (branched) inverses of $F^{\ell}$.
Since $F^{-\ell}$ is pre-composed by $h$ it is meaningful even when restricted to $V$.

So, each multivalued map $\phi_\ell$
assigns to a point $x\in U_0$ some  collection of its preimages under $f^n$, with $n=n_\ell=\ell +1$, all of them being in $U_0$, even in $U_\ell$.

\smallskip

For $x\in U_0$ and
$\ell\in\mathbb N\cup\{0\}$ we write
$\sum |\phi_\ell'(x)|$
to denote the summation of derivatives which runs over all branches of the multivalued map $\phi_\ell$.
If a critical point $c$ and its $f$-image are met then we can put in this sum $\infty$, as inverse
of forward derivative 0.  In fact it does not matter since we can restrict to $x$ not post-critical.

For every  $\ell\in\mathbb N \cup\{0\}$ and $x\in U_0$ denote 
$$
L^*_\ell (f,x)=\log \sum |\phi_\ell'(x)|,
$$
where the summation runs over all branches of the multivalued map $\phi_\ell$.
Denote also by $\underline\ell$ a sequence $\underline \ell=(\ell_1,\dots \ell_k)$, and put
$$
n_{\underline \ell}=n_{\ell_1}+\dots +n_{\ell_k}.
 $$

Finally, for a sequence  $\underline \ell=(\ell_1,\dots \ell_k)$, denote

\begin{equation}\label{L*}
L_{\un\ell}(f,x):= L^*_{\ell_k}(f, y_{k-1})+...+L^*_{\ell_2}(f,y_1)+ L^*_{\ell_1}(f,x),
\end{equation}
 where $y_1=\phi_1(x), ..., y_{k-1}=\phi_k(y_{k-2})$ (remember that the maps $\phi_j$ are multivalued, so we consider sums over their values).
The star $ ^*$ means we consider only preimages along first $h$ and next $F^{-\ell_i}$.
In other words, more formally,
$$
L_{\un\ell}(f,x)=\log\sum|\phi'_{\underline \ell}(x)|,
$$
where the summation runs over all branches of the multivalued function
$$\phi_{\underline \ell}=\phi_{\ell_k}\circ\phi_{l_{k-1}}\circ\dots \circ \phi_{\ell_1}.$$

\smallskip

It is important to note that all the multivalued functions $\phi_{\un\ell}$ are mutually distinct, compare  \emph{free Iterated Function System} in
\cite{InoRiv}.
Indeed, let $\un\ell\not=\un\ell'$,
but $n_{\un\ell}=n_{\un\ell'}$.
Let $i$ be the first integer such that $\ell_i\not=\ell'_i$. Then, supposing that $n_{\ell_i}>n_{\ell'_i}$
we compose in $\un\ell'$ after $\ell'_i$ with $h$ with range in $V$, whereas in $\un\ell$ still within $\ell_i$  with $F^{-1}$
having range
$U_1$. Further compositions by branches of $f^{-1}$ preserve distinction.

\medskip

Denote by $\Sigma^*$ the set of all finite sequences $\underline\ell=(\ell_1,\dots \ell_k)$, $k\ge 1$.
For an arbitrary $x\in U_0$
denote
$$
\Lambda_N(x)=\sum_{\un\ell\in\Sigma^*: n_{\un\ell}=N} \exp L_{\un\ell}(f,x).
$$

 The star $ ^*$ again means we consider only preimages along first $h$ next $F^{-1}$.

\begin{proposition}\label{prop:est_pressure} For every $N\ge 1$ and non-exceptional, in particular not post-critical, $x\in U_0$,
$$
P(1,f; x)\ge
\liminf_{N\to\infty}
\frac{1}{N}
\log \Lambda_N(x).
$$
\end{proposition}

\begin{proof}
In view of Proposition \ref{prop:bowen}  it is sufficient to prove
$$
\sum_{y\in f^{-N}(x)}\frac{1}{|(f^N)'(y)|} \ge \Lambda_N(x)
$$
This is however obvious, because on the left hand side all $y\in f^{-N}(x)$ appear, whereas on the right hand side
only selected ones.
\end{proof}

\begin{proposition}\label{prop:pressure_large}
Let $\Lambda_N:=\inf_{x\in U_0}\Lambda_N(x).$ Then
$$
\liminf_{N\to\infty}\frac{1}{N}\log \Lambda_N>0.
$$
\end{proposition}

\begin{proof}
Put $a:=\inf_{x\in U_0}\left (\sum |h'(x)|\right )$, and $b:=\exp(-C_0)$ where $C_0$ comes from Lemma ~\ref{lem:key}.

Consider all  sequences $\underline\ell=(\ell_1,\dots \ell_k)$ such that $n_{\underline \ell}=N$, for each  $k\le N$. The number of such sequences
can be calculated
in the following way:

For each $k\le N$, there are ${N-1}\choose {k-1}$ ways of choosing $k-1$ positions $m_1,\dots m_{k-1}$ from the sequence $\{1,\dots N-1\}$.\footnote{The $\underline\ell$ has length $N$, but its first place is occupied by the beginning of the first block. So the only choices for the beginnings of remaining $k-1$ blocks are the remaining $N-1$ places.}
Having these positions chosen, we assign to them the values $\ell_1=m_1-1$, $\ell_2=(m_2-m_1)-1$, $\ell_k=(N-m_{k-1})-1$ and  the (multivalued) map
$\phi_{\underline\ell}$ with  $\underline \ell=(\ell_1,\dots\ell_k)$ and $n_{\underline\ell}=(\ell_1+1)+(\ell_2+1)+\dots +(\ell_k+1)=N$.

\smallskip

 Then we have the estimate

 $$\exp L_{\un\ell}(f,x)\ge a^kb^k$$
 and, consequently,

 $$\sum_{\un\ell\in\Sigma^*: n_{\un\ell}=N} \exp L_{\un\ell}(f,x)\ge \sum_{k=1}^{N}{{N-1}\choose {k-1}}(ab)^k =ab(1+ab)^{N-1}.$$

\end{proof}

{\bf Remark.}\
Let us note that a calculation in a similar spirit appeared in \cite{InoRiv} and in \cite{ms}.
The authors used there a method of generating functions (some power series with coefficients related to a value similar to $\Lambda_N$).
Our case is simpler, in a sense that the set of admissible values $n_\ell$ forms an arithmetic sequence. So, the straightforward calculation provided above is sufficient.\footnote{However in Section \ref{final}, Comment 1, we use the general, power series, version.}


\bigskip

Obviously, combining Proposition~\ref{prop:est_pressure} and Proposition~\ref{prop:pressure_large}, together with Proposition ~\ref{prop:bowen}, we conclude the proof of Proposition~\ref{prop:dim_large}.

Indeed,  Proposition~\ref{prop:est_pressure} and Proposition~\ref{prop:pressure_large} imply immediately that for every non-exceptional point $x\in U_0$ we have that
$$P(1,f;x)>0,$$
which implies (see Proposition~\ref{prop:bowen}) that $P(1,f)>0$.
Moreover,  Proposition~\ref{prop:bowen} asserts that
$${\rm HD}_{\hyp}(J(f))=\inf\{t>0: P(t,f)\le 0\}.$$
Since the function  $t\mapsto P(t, f)$ is continuous and non--increasing for $t\ge 0$,
we conclude immediately that
$${\rm HD}_{\hyp}(J(f))>1.$$

\end{proof}

\section{Example}\label{sec:example}

To complete the answer to the question of Christopher Bishop, it remains to ask whether there exists a polynomial with disconnected Julia set, and a connected component of $J(f)$ being an analytically embedded interval.

In the following proposition we provide a family of maps with this property.

\begin{prop}\label{prop:interval_components}
Consider a family of   cubic polynomials
\begin{equation}\label{eq:example}
f_{\varepsilon,\beta}(z):=
\varepsilon z^3+z^2-\beta,
\end{equation}
with $\varepsilon$  and $\beta$ real, except $\varepsilon=0$ for which the polynomials are quadratic.
Then there exists an analytic curve $\Gamma$ of parameters $(\varepsilon,\beta)$, which is the graph of an analytic function $\varepsilon\mapsto \beta(\varepsilon)$ for $\varepsilon\approx 0$,
passing through the  point $(0,2)$, and such that for every $(\varepsilon, \beta)\in\Gamma$ with $\varepsilon> 0$ the Julia set of $f_{\varepsilon,\beta}$ is disconnected,
and contains infinitely many components being analytic arcs.

\end{prop}
\begin{proof}
We start with the quadratic Chebyshev polynomial $f_{0,2}(z)=z^2-2$. Its Julia set is just the interval $I:=[-2,2]$.

Consider now the  polynomials $f_{\varepsilon,\beta}$
with $\varepsilon$ real and close to $0$, and $\beta$ real and close to $2$.
The parameters $(\varepsilon,\beta)=(0,2)$ will be called \emph{initial parameters}.

The map $f_{\varepsilon,\beta}$ has a real repelling fixed point, denoted by $p_{\varepsilon, \beta}$, close to $p_{0,2}=2$, and $0$ is a (not moving) critical point of $f_{\varepsilon,\beta}$.

\begin{lemma}
There exists an analytic curve $\Gamma$  of parameters $(\varepsilon, \beta)$, passing through the initial parameters $(0,2)$ for which

$$f^2_{\varepsilon,\beta}(0)=p_{\varepsilon,\beta}.$$
Moreover $\partial \Gamma/\partial \beta \not=0$ at the initial parameters and in consequence
$\Gamma$ is the graph of an analytic function $\varepsilon \mapsto \beta(\varepsilon)$.
\end{lemma}
\begin{proof}
This is a straightforward calculation.
Denoting
$\gamma(\varepsilon, \beta)=f^2_{\varepsilon,\beta}(0)$,
we have $\gamma(\varepsilon, \beta)=-\varepsilon\beta^3+\beta^2-\beta$,
so
$$
{\rm grad}(\gamma)(0,2)=\overrightarrow{(-8,3)}.
$$

Denoting by $p(\varepsilon,\beta)$ the fixed point of $f_{\varepsilon,\beta}$ close to $2$,
we calculate (differentiating  the implicit function)
$${\rm grad}(p)(0,2)=\overrightarrow{(-8/3, 1/3)}.$$

Thus, $$
{\rm grad}(\gamma-p)(0,2)\neq \overrightarrow{(0,0)}
$$
and therefore, there exists a smooth curve of parameters $(\varepsilon, \beta)$, passing through the initial parameters $(0,2)$ for which
$$
f^2_{\varepsilon,\beta}(0)=p_{\varepsilon,\beta}.
$$
In particular $(\partial(\gamma-p)/\partial \beta) (0,2) = 3-1/3 \not=0$
hence by implicit function theorem
$\Gamma$ is indeed the graph of an analytic function $\beta(\varepsilon)$.
\end{proof}
Notice that $f_{\varepsilon,\beta}(0)=-\beta$. The interval $I_{\varepsilon,\beta}=[-\beta, p_{\varepsilon, \beta}]$ is thus invariant under
$f_{\varepsilon,\beta}$ and the map
$f_{\varepsilon,\beta}: I_{\varepsilon,\beta}\to I_{\varepsilon,\beta}$ is
two-to-one, with critical point at $0$, hence topologically conjugate to the quadratic Chebyshev polynomial $z^2-2$ on $I$ as above.

\medskip

Suppose from now on that $\varepsilon>0, \varepsilon\approx 0$ and $\beta=\beta(\varepsilon)$. Let us note that the Julia set $J(f_{\varepsilon,\beta})$ is not connected, as the trajectory of the second critical point $c=-\frac{2}{3\varepsilon}$ tends to infinity.
The latter holds since $f_{\varepsilon,\beta}(-\frac{2}{3\varepsilon})>p_{\varepsilon,\beta}$
 is repelled to infinity under further action of $f_{\varepsilon,\beta}$.

Denote by $C$ the connected component
of $J(f_{\varepsilon, \beta})$ containing $I_{\varepsilon,\beta}$, so in particular it is fixed under $f_{\varepsilon,\beta}$ and is non-trivial.

By  Lemma~\ref{lem:poly_like}, there exist connected  neighbourhoods of $ C$,
$$ C\subset U_1\subset \overline U_1\subset U_0,$$ such that the map  $F:={(f_{\varepsilon,\beta})}_{|U_1}:U_1\to U_0$ is a polynomial-like map,
and $ C$ is its filled-in Julia set. Since the Julia set of $f_{\varepsilon,\beta}$ is disconnected, then by  Remark 2. after
Proof of Lemma \ref{lem:poly_like}, the degree of
$F$
is not maximal, so it is equal to $2$.
By \cite[Theorem 1 (The Straightening Theorem)]{dh}, see also more recent \cite[Theorem 7.4]{bf},  $F$
is quasiconformally conjugate to a true quadratic
polynomial.  Therefore,  preimages of every point in $ C$ are dense in $
 C$. But the preimages under  $F$ of points from
$I_{\varepsilon,\beta}$ remain in $I_{\varepsilon,\beta}$, which implies that
$$ C=I_{\varepsilon,\beta}.$$

So, each map $f_{\varepsilon,\beta}$ with $(\varepsilon, \beta)\in\Gamma$ is  a polynomial of degree 3 with disconnected Julia set,  for which the Julia set has an invariant component  being a true interval on which the degree of the map is equal to 2.
So, for each such map the filled-in Julia set, here equal to the Julia set,  has a collection of countably many non-trivial components, each of them being an analytic arc; this collection is formed by the invariant analytic arc and all its preimages under the  iterates of $f_{\varepsilon,\beta}$.

Note also that these are the only non-trivial components of the filled-in Julia set $K(f_{\varepsilon,\beta})$.
Indeed, by Theorem A all non-trivial components are eventually periodic and by Remark 1 after Lemma~\ref{lem:poly_like}
every non-trivial periodic component of the filled-in Julia set has to contain a critical point in its orbit.
In our situation,  there are two critical points; one of them is escaping, and the other one is already contained in the invariant interval $I_{\varepsilon,\beta}$.

\end{proof}

 \medskip


\section{Final remarks and questions}\label{final}

\

\noindent COMMENT 1. \ For $f$ as in Proposition \ref{prop:dim_large} and $\partial C$ the boundary of a connected components $C$ of filled-in Julia set
$K(f)$, one can ask under what assumptions it holds

\begin{equation}\label{strict_ineq}
\HD(\partial C)< \HD(J(f)).
\end{equation}

It does not hold in general. It may happen that the component $C$ has empty interior (so is the Julia set of the map $F$ introduced in Lemma~\ref{lem:poly_like}), but it satisfies
\ $\HD(J(F))=HD(J(f))=2$ and even $\HD_{\rm hyp}(J(F))=2$.
Indeed,
just in the example $f_{\varepsilon,\beta}(z)=\varepsilon z^3 + z^2-\beta$ with $\varepsilon>0$ (but $\beta\not=\beta(\varepsilon)$), we find $\varepsilon\approx 0$ and complex $\beta \approx 2$
such that $\HD_{\rm hyp}(J(F))=2$, by finding\footnote{See \cite{dh} or \cite[Theorem 7.8]{bf}.
The family $\mathcal F$ of quadratic-like maps $f_{\varepsilon,\beta}$ for an arbitrary small $\varepsilon$ for $|\beta|<3$ and $|z|<3$ is Mandelbrot-like, as a small perturbation of the family
$z^2-\beta$. So there is 1-to-1 (homeomorphic) correspondence by hybrid (in particular quasiconformal) equivalences of its elements with connected filled-in Julia sets, to polynomials $z^2+c$ with $c$ in the Mandelbrot set. This is a
parameter version of Douady-Hubbard's Straightening Theorem. }
parameters so that the quadratic polynomial
hybrid equivalent to $F$ satisfies this, \cite{Shishikura}.
Note that a quasiconformal conjugacy cannot drop neither Hausdorff nor hyperbolic dimension down from 2, see  \cite{Astala}.

\medskip

Proving Proposition \ref{prop:dim_large} we proved and used the fact
that $f$  is ``almost hyperbolic'', here: topologically (analytically) conjugate to a $\pm$ Chebyshev polynomial
on a neighbourhood of $C$.
An even easier case is where $f$ is hyperbolic (expanding) on $C$. Then Lemma \ref{lem:key}, leading to \eqref{strict_ineq},
holds easily.
In fact \eqref{strict_ineq} holds 
under more general assumptions, namely if $F$ is \emph{non-uniformly hyperbolic} here in the version that it satisfies \emph{Topological Collet-Eckmann}, TCE, see \cite{PRS}. One out of several definitions equivalent to TCE,
called \emph{Backward Exponential Shrinking},
says:

There exist $\lambda>1$ and $r>0$  such that for every $x \in \partial C=J(F)$,
every $n > 0$ and every connected component $W_n$ of $F^{-n}({\mathbb D}(x, r))$ for the Euclidean disc
${\mathbb D}(x,r)$ with radius $r$ and origin at $x$,
   it holds     $\diam (W_n) \le \lambda^{-n}$.

 \begin{proof}[Sketch of proof of \eqref{strict_ineq} for $C$ and $F$ satisfying TCE] Assume for simplification

 \noindent that $\partial C $ contains just one $f$-critical point. We shall apply an analogon of
 Lemma~\ref{lem:key},
 where $F$ is  replaced by a
\emph{Canonical induced map} $G$,
 which is a  map of return (not first return !) to a \emph{nice} set $\mathcal W$ for iteration of $F$, \ $G(x):=F^{m(x)}(x)$. We refer here to  a theory developed by the first author and J. Rivera-Letelier in \cite{PRL2} and
 \cite{PR-L}. In particular see Section 3 of the latter paper. $G^{-1}$ forms an Iterated Function System with infinitely many branches from $\mathcal W$ into itself, with bounded distortion.
 Consider the pressure
 \begin{equation}\label{return_pressure}
 \mathcal P(t,p,G):=
 \lim_{n\to\infty}\frac1n \log
 \sum_{v\in G^{-n}(w)} \exp S_{n,G}(\psi_G)(v)
 \end{equation}
 for
 $\psi_G:=-t\log (|G'|-pm)$
 where $S_{n,G}(\psi_G):=\sum_{j=0}^{n-1}\psi_G \circ G^j$,
 not depending on $w$ due to bounded distortion.
 Compare Proposition \ref{prop:bowen} and \eqref{eq:L_n}.
 Now we apply \cite[Subsection 7.2]{PR-L} which says that $\mathcal P(t,P(t,F), G)=0$
 for $t=t_0$ being zero of $P(t,F)$
 Next denote $\psi:= -t_0 \log|F'| $.
 Let $h$
be an additional branch of $f^{-m_0}$ from $\mathcal W$ into itself for some $m_0\in\N$, compare
Lemma~\ref{lem:other_branch}.\footnote{The nice set $\mathcal W$ is a small neighbourhood of the set
$\Crit(F,J)$ of all
$F$-critical points in $J(F)$. If $\#\Crit(F,J)>1$, it has more than one component.  So this needs some care. In place of the IFS one applies a Graph Directed Markov System  and defines $\mathcal P$
appropriately, see \cite{MU}. The branch of $h$ above means a branch for each component of $\mathcal W$,
where it is comfortable to choose  $m_0$ common.
Such a choice is possible due to the \emph{topological exactness} of $f$ on $J(f)$, what means that for every non-empty set $A\subset J(f)$ open in
$J(f)$ there exists $n\in\N$ such that $f^n(A)=J(f)$. Compare \cite[Proof of Lemma 4.1]{PRL2}.}
Now,
for each $\ell\in\mathbb N$ consider the multivalued function $\phi_\ell:=G^{-\ell} \circ h$.
We obtain the key estimate
 similar to \eqref{multivalued} in Proof of Proposition \ref{prop:dim_large}
 \begin{equation}\label{key2}
\log \sum_{y\in \phi_\ell (x)}\exp S_{m_\ell}(\psi)(y) >-C_0
\end{equation}
for  $m_\ell:={m(y)+m(G(y))+... m(G^{\ell-1}(y))+m_0}$.
Here $S_n$ means $S_{n,F}:=\sum_{j=0}^{n-1}\psi \circ F^j$.
The estimate \eqref{key2} follows from \cite[Section 6]{MU} and \cite[Subsection 4.3. \emph{Existence}]{PR-L}; $G$ is hyperbolic, but with infinitely many branches.
For each $\ell$ define the series

\medskip

$\Phi_\ell (s):=\sum_{N=1}^\infty A_{\ell, N} s^N$ for complex $s$,
with coefficients
$$
A_{\ell,N}=\sum_{y\in \phi_\ell(x), m_{\ell}(y)=N} \exp S_N \psi(y).
$$
Notice that \eqref{key2} yields
$$
\Phi_\ell (1)=\sum_N  A_{\ell,N}> \exp -C_0.
$$
So, for every $L$ there exists $N$,
such that for every $\ell<L$,
$$
\sum_{n\le N} A_{n,\ell} \ge \frac12 \exp -C_0.
$$
So, for $\Phi(s):=\sum_\ell \Phi_\ell(s)$ we get
$$
\Phi(1)\ge \sum_{n\le N, \ell<L} A_{n,\ell} \ge \frac12 L \exp -C_0>2
$$
for $L$ large enough. Hence also $\Phi(s)>2$ for some $s<1$.

So, $\Phi(s)+\Phi(s)^2+ ... =\infty$, hence the radius of convergence of  $\Phi$ is at most $s$.
So $P(t_0, f )>P(t_0, F)=0$. 
Hence the (unique) zero of the function $P(t, f)$, equal to $\HD_{\rm hyp}(J(f))$ \ ($=\HD(J(f))$ due to TCE),
is strictly larger than $t_0=\HD_{\rm hyp}(J(F))=\HD(J(F))$.
\end{proof}

Notice that in the situation of Propositon 5 , given $\ell$,\  $\Phi_\ell=A_{m_\ell,\ell} s^{m_\ell}$, where the time of each backward branch for its iterate of $f$ is fixed, equal to $m_\ell$,  whereas here it is not.

\medskip



{\bf Question.} How TCE can be weakened so that \eqref{strict_ineq} still holds?

\

\

\noindent COMMENT 2. \ Theorems
\ref{thm:main}
and
\ref{thm:HDabove}
hold for $f$ being polynomial-like, so in the situation more general than for $f$ being a polynomial. Proofs are the same.\footnote{We do not know whether we can refer to Theorems \ref{thm:main}
and
\ref{thm:HDabove} for polynomials using straightening, since
quasi-conformal homeomorphisms may in general change Hausdorff dimension larger than 1 to equal to 1.}

\

\

\noindent COMMENT 3. \ A positive answer to the following conjecture would generalize Theorem \ref{thm:HDabove} to rational maps:

 \begin{conjecture}\label{conj:main}
Let $f:\widehat{\mathbb C}\to\widehat{\mathbb C}$ be a rational function  of degree at least 2. Assume that  the Julia set $J(f)$ is connected. Then $f$ is either a finite Blaschke product
in some holomorphic coordinates or a quotient of a Blaschke product by a rational function of degree 2,
or $\HD(J(f))>1$.
\end{conjecture}

If $f$ has an attracting periodic orbit, then the positive answer is given by Theorem B applied to
$\Omega$ being the immediate basin of attraction to a point of the orbit for an iterate of $f$. Indeed $\Omega$ is simply connected by the connectivity of $J(f)$. More precisely, either $\HD_{\hyp}(J(f))>1$ or due to
\cite[Theorem A]{feliks} it is like in the assertion of the conjecture.

We believe that the same holds if $f$ has a parabolic periodic orbit.

\medskip

Another possibility is that the Fatou set consists only of Siegel discs and their preimages
for iterates of $f$.

Before discussing this, consider  $f=P_\alpha$, a quadratic polynomial ${\rm e}^{2\pi i\alpha}z +z^2$
such that
$\alpha$ is real irrational of sufficiently high type, namely all the coefficients
in its continued fraction
expansion are larger than certain constant.  Then, if $\alpha\in {\mathcal B} \setminus {\mathcal H}$, i.e. satisfying Brjuno
condition, (equivalent to the linearizability at 0, yielding a Siegel disc $S$), but not satisfying Herman condition (in particular not
Diophantine), the Siegel disc is \emph{hairy} outside. Namely there is a Cantor bouquet of lines
in $J(P_\alpha)$ having together Hausdorff dimension 2. See \cite{cdy} and references therein.

If $\alpha$ is Diophantine, of constant type, namely all the coefficients in its continued fraction expansion are smaller than a constant, then  $\HD(\partial S)>1$, see \cite{gj}.

\smallskip

In those situations however, $P_\alpha$ are polynomials so they have basins of attraction to infinity in their Fatou sets and Theorem 2
proves the Conjecture.
But maybe for a rational map $f$ the above theory can also be applied, so
in the ``Siegel'' cases above, Conjecture \ref{conj:main} holds true.

\medskip

In the remaining case of Herman rings, the situation is expected to be similar to Siegel's.

\medskip

Notice that a positive answer to Conjecture \ref{conj:main} follows immediately from the Comment 2,
for $f$ being renormalizable.
\emph{Renormalizable} means here, that there exist an integer $r\ge 2$, a compact connected
set $C$ having more than one point, intersecting $J(f)$, and its neighbourhoods $U$ and $U'$,  such that $F=f^r|_{U'}:U'\to U$ is polynomial-like with the filled-in Julia set equal to $C$.
Unlike before, we do not assume here that $\partial C$ is a connectivity component of $J(f)$.
See examples of such renormalizations for rational Newton maps in \cite{DS}.

\end{document}